\documentclass[12pt]{article}

\usepackage{amsmath, epsfig, cite}
\usepackage{amssymb}
\usepackage{amsfonts}
\usepackage{latexsym}
\usepackage{amsthm}

\newtheorem{thm}{Theorem}[section]

\newtheorem{cor}[thm]{Corollary}

\newtheorem{lem}[thm]{Lemma}

\newcommand{\binq}[2]{\genfrac{[}{]}{0mm}{0}{#1}{#2}}
\newcommand{\tbnq}[2]{\genfrac{[}{]}{0mm}{1}{#1}{#2}}
\numberwithin{equation}{section}

\renewcommand{\thefootnote}{}

\begin{document}

\begin{center}
{\large\bf Some $q$-Supercongruences for the truncated $q$-trinomial
coefficients

 \footnote{The work is supported by the National Natural Science Foundations of China (No. 12071103).}}
\end{center}

\renewcommand{\thefootnote}{$\dagger$}

\vskip 2mm \centerline{Chuanan Wei}
\begin{center}
{School of Biomedical Information and Engineering,\\ Hainan Medical University, Haikou 571199, China
\\Email address: weichuanan78@163.com}
\end{center}


\vskip 0.7cm \noindent{\bf Abstract.} In 1987, Andrews and Baxter
introduced six kinds of $q$-trinomial coefficients
in exploring the
solution of a model in statistical mechanics. In
 this paper, we give some  $q$-supercongruences for the truncated forms of these polynomials.



\section{Introduction}
Define the trinomial coefficient $\big(\!\binom{n}{m}\!\big)$ to be
\[(1+x+x^2)^n=\sum_{m=-n}^n\bigg(\!\binom{n}{m}\!\bigg)x^{m+n}.\]
It is well known that there are the following two simple formulas
 (cf. \cite[P. 43]{Sills}):
\begin{align*}
&\qquad\bigg(\!\!\binom{n}{m}\!\!\bigg)=\sum_{k=0}^{n}\binom{n}{k}\binom{n-k}{m+k},\\[2mm]
&\bigg(\!\!\binom{n}{m}\!\!\bigg)=\sum_{k=0}^{n}(-1)^k\binom{n}{k}\binom{2n-2k}{n-m-k}.
\end{align*}

In 1819, Babbage \cite{Babbage} proved the interesting congruence:
for any prime $p\geq3$,
\begin{align}\label{Babbage}
\binom{2p-1}{p-1}\equiv1\pmod{p^2}.
\end{align}
Decades later, Wolstenholme \cite{Wolstenholme} told us that the
above formula is true modulo $p^3$ for any prime $p\geq5$. In 1952,
Ljunggren showed the generalization of Wolstenholme's result (cf.
\cite{Granville}):
\begin{align}
\binom{ap}{bp}\equiv\binom{a}{b}\pmod{p^3}.
 \label{Ljunggren}
\end{align}
Inspired by the works just mentioned,  it is natural to consider
supercongruence associated with the trinomial coefficient
$\big(\!\binom{ap}{bp}\!\big)$.

Let $[r]$ be the $q$-integer $(1-q^r)/(1-q)$ and define the
$q$-binomial coefficient $\tbnq{n}{m}$ by
\begin{align*}
\binq{n}{m}=\binq{n}{m}_q=
\begin{cases} \displaystyle \frac{(1-q^n)(1-q^{n-1})\cdots(1-q^{n-m+1})}{(1-q)(1-q^2)\cdots(1-q^m)}, &\text{if $0\leq m\leq n$;}\\[5pt]
 0, &\text{otherwise.}
\end{cases}
\end{align*}
All over the paper, let $\Phi_n(q)$ stand for the $n$-th cyclotomic
polynomial in $q$:
\begin{equation*}
\Phi_n(q)=\prod_{\substack{1\leqslant k\leqslant n\\
\gcd(k,n)=1}}(q-\zeta^k),
\end{equation*}
where $\zeta$ is an $n$-th primitive root of unity. It is surprising
that Andrews \cite{Andrews-a} gave a $q$-analogue of
\eqref{Babbage}:
\begin{align*}
\binq{2p-1}{p-1}\equiv q^{\frac{p(p-1)}{2}}\pmod{[p]^2},
\end{align*}
where $p\geq3$ is any prime. In 2019, Straub \cite{Straub}
discovered a $q$-analogue of \eqref{Ljunggren}:
\begin{align*}
\binq{an}{bn}\equiv
\binq{a}{b}_{q^{n^2}}+(a-b)b\binom{a}{b}\frac{1-n^2}{24}(1-q^n)^2\pmod{\Phi_n(q)^3},
\end{align*}
where $a,b$ are nonnegative integers and $n$ is a positive integer.

In 1987, Andrews and Baxter \cite{Andrews-b} introduced six kinds of
$q$-trinomial coefficients in exploring the solution of a model in
statistical mechanics. They can be laid out as follows:
\begin{align*}
&\bigg(\!\!\binom{n}{m}\!\!\bigg)_q=\sum_{k=0}^{n}q^{k(k+m)}\binq{n}{k}\binq{n-k}{m+k},
\\[2mm]
&\tau_0(n,m,q)=\sum_{k=0}^{n}(-1)^kq^{nk-\binom{k}{2}}\binq{n}{k}\binq{2n-2k}{n-m-k},
\\[2mm]
&T_0(n,m,q)=\sum_{k=0}^{n}(-1)^k\binq{n}{k}_{q^2}\binq{2n-2k}{n-m-k},
\\[2mm]
&T_1(n,m,q)=\sum_{k=0}^{n}(-q)^k\binq{n}{k}_{q^2}\binq{2n-2k}{n-m-k},
\end{align*}
\begin{align*}
&t_0(n,m,q)=\sum_{k=0}^{n}(-1)^kq^{k^2}\binq{n}{k}_{q^2}\binq{2n-2k}{n-m-k},
\\[2mm]
&t_1(n,m,q)=\sum_{k=0}^{n}(-1)^kq^{k(k-1)}\binq{n}{k}_{q^2}\binq{2n-2k}{n-m-k}.
\end{align*}
Recently, Liu \cite{Liu-a} determined $\big(\!\binom{n}{0}\!\big)_q$
and $\big(\!\binom{2n}{n}\!\big)_q$ modulo $\Phi_n(q)^2$.  Chen, Xu
and Wang \cite{Chen} further studied
$\big(\!\binom{an}{an-n}\!\big)_q$ modulo $\Phi_n(q)^2$.
 There are more $q$-analogues of supercongruences in the
literature. We refer the reader to
\cite{Guo-adb,GS,GS20c,GuoZu-a,GuoZu-b,Li-Wang,LJ,LP,LW,LW-a,WY,WY-a,Wei-a,Zu19}.

Motivated by the works just mentioned, we shall establish the
following theorem.

\begin{thm}\label{thm-a}
Let $\big(\!\binom{an}{bn}\!\big)_q'$ denote the truncated form of
the $q$-trinomial coefficient $\big(\!\binom{an}{bn}\!\big)_q$:
\begin{align*}
\bigg(\!\!\binom{an}{bn}\!\!\bigg)_q'=\sum_{k=0}^{\lfloor
n/2\rfloor}q^{k(k+bn)}\binq{an}{k}\binq{an-k}{bn+k},
\end{align*}
where $a,b,n$ are positive integers subject to $a>b$ and $\lfloor
x\rfloor$ is the integral part of a real number $x$. Then, modulo
$\Phi_n(q)^2$,
\begin{align}
\bigg(\!\!\binom{an}{bn}\!\!\bigg)_q'
\equiv\binq{an}{bn}\Big\{1-(a-b)\big(1-\theta_n(q)\big)\Big\},
\label{eq:wei-a}
\end{align}
where
\begin{align*}
&\theta_n(q) \equiv
\begin{cases} \displaystyle (-1)^m(1+q^m)q^{m(3m-1)/2}, &\text{if $n=3m$;}\\[5pt]
 (-1)^mq^{m(3m+1)/2}, &\text{if $n=3m+1$;}\\[5pt]
 (-1)^{m+1}q^{(m+1)(3m+2)/2}, &\text{if $n=3m+2$.}
\end{cases}
\end{align*}
\end{thm}

Choosing $n=p$ and then letting $q\to1$ in the above theorem, we
obtain the supercongruence after using \eqref{Ljunggren}.

\begin{cor}\label{cor-a}
Let $\big(\!\binom{ap}{bp}\!\big)'$ represent the truncated form of
the trinomial coefficient $\big(\!\binom{ap}{bp}\!\big)$:
\begin{align*}
\bigg(\!\!\binom{ap}{bp}\!\!\bigg)'=\sum_{k=0}^{(p-1)/2}\binom{ap}{k}\binom{ap-k}{bp+k},
\end{align*}
where $a,b$ are positive integers with $a>b$ and $p$ is an old
prime. Then
\begin{align*}
\bigg(\!\!\binom{ap}{bp}\!\!\bigg)' \equiv\binom{a}{b}\pmod{p^2}.
\end{align*}
\end{cor}

Very recently, Liu and Qi \cite{Liu-b} made sure of
$\tau_0(an,an-n,q)$, $T_0(an,an-n,q)$, and $T_1(an,an-n,q)$ modulo
$\Phi_n(q)^2$. We shall establish the following five theorems.

\begin{thm}\label{thm-b}
Let $\tau_0(an,bn,q)'$ stand for the truncated form of the
$q$-trinomial coefficient $\tau_0(an,bn,q)$:
\begin{align*}
\tau_0(an,bn,q)'=\sum_{k=an-bn-\lfloor n/2\rfloor}^{an-bn}
(-1)^kq^{ank-\binom{k}{2}}\binq{an}{k}\binq{2an-2k}{an-bn-k},
\end{align*}
where $a,b,n$ are positive integers satisfying $a>b$. Then, modulo
$\Phi_n(q)^2$,
\begin{align}
\tau_0(an,bn,q)' &\equiv(-1)^{an-bn}q^{(ab-bn)(an+bn+1)/2}
\notag\\[2mm]
&\quad\times\binq{an}{bn}\Big\{1-(a-b)\big(2-\theta_n(q)-\vartheta_n(q)\big)\Big\},
\label{eq:wei-b}
\end{align}
where
\begin{align*}
&\vartheta_n(q) \equiv
\begin{cases} \displaystyle (-1)^m(1+q^{2m})q^{m(3m-5)/2}, &\text{if $n=3m$;}\\[5pt]
 (-1)^mq^{m(3m+1)/2}, &\text{if $n=3m+1$;}\\[5pt]
 (-1)^{m+1}q^{(m-1)(3m+2)/2}, &\text{if $n=3m+2$.}
\end{cases}
\end{align*}
\end{thm}

\begin{thm}\label{thm-c}
Let $T_0(an,bn,q)'$ denote the truncated form of the $q$-trinomial
coefficient $T_0(an,bn,q)$:
\begin{align*}
T_0(an,bn,q)'=\sum_{k=an-bn-\lfloor n/2\rfloor}^{an-bn}
(-1)^k\binq{an}{k}_{q^2}\binq{2an-2k}{an-bn-k},
\end{align*}
where $a,b,n$ are positive integers subject to $a>b$. Then, modulo
$\Phi_n(q)^2$,
\begin{align}
T_0(an,bn,q)' &\equiv(-1)^{an-bn}
\binq{an}{bn}_{q^2}\Big\{1-2(a-b)\big(1-\theta_n(q)\big)\Big\}.
\label{eq:wei-c}
\end{align}
\end{thm}

\begin{thm}\label{thm-d}
Let $T_1(an,bn,q)'$ represent the truncated form of the
$q$-trinomial coefficient $T_1(an,bn,q)$:
\begin{align*}
T_1(an,bn,q)'=\sum_{k=an-bn-\lfloor n/2\rfloor}^{an-bn}
(-q)^k\binq{an}{k}_{q^2}\binq{2an-2k}{an-bn-k},
\end{align*}
where $a,b,n$ are positive integers with $a>b$. Then, modulo
$\Phi_n(q)^2$,
\begin{align}
T_1(an,bn,q)' &\equiv(-q)^{an-bn}
\binq{an}{bn}_{q^2}\Big\{1-2(a-b)\big(1-\vartheta_n(q)\big)\Big\}.
\label{eq:wei-d}
\end{align}
\end{thm}

\begin{thm}\label{thm-e}
Let $t_0(an,bn,q)'$ stand for the truncated form of the
$q$-trinomial coefficient $t_0(an,bn,q)$:
\begin{align*}
t_0(an,bn,q)'=\sum_{k=an-bn-\lfloor n/2\rfloor}^{an-bn}
(-1)^kq^{k^2}\binq{an}{k}_{q^2}\binq{2an-2k}{an-bn-k},
\end{align*}
where $a,b,n$ are positive integers satisfying $a>b$. Then, modulo
$\Phi_n(q)^2$,
\begin{align}
t_0(an,bn,q)' &\equiv(-1)^{an-bn}q^{(an-bn)^2}
\binq{an}{bn}_{q^2}\Big\{1-2(a-b)\big(1-\theta_n(q^{-1})\big)\Big\}.
\label{eq:wei-e}
\end{align}
\end{thm}

\begin{thm}\label{thm-f}
Let $t_1(an,bn,q)'$ denote the truncated form of the $q$-trinomial
coefficient $t_1(an,bn,q)$:
\begin{align*}
t_1(an,bn,q)'=\sum_{k=an-bn-\lfloor n/2\rfloor}^{an-bn}
(-1)^kq^{k(k-1)}\binq{an}{k}_{q^2}\binq{2an-2k}{an-bn-k},
\end{align*}
where $a,b,n$ are positive integers subject to $a>b$. Then, modulo
$\Phi_n(q)^2$,
\begin{align}
t_1(an,bn,q)' &\equiv(-1)^{an-bn}q^{(an-bn)(an-bn-1)}
\binq{an}{bn}_{q^2}\Big\{1-2(a-b)\big(1-\upsilon_n(q^{-1})\big)\Big\}.
\label{eq:wei-f}
\end{align}
\end{thm}

Setting $n=p$ and then letting $q\to1$ in Theorems
\ref{thm-b}-\ref{thm-f}, we get the supercongruence after utilizing
\eqref{Ljunggren}.

\begin{cor}\label{cor-b}
Let $\big(\!\binom{ap}{bp}\!\!\big)^{*}$ represent the truncated form
of the trinomial coefficient $\big(\!\binom{ap}{bp}\!\big)$:
\begin{align*}
\bigg(\!\!\binom{ap}{bp}\!\!\bigg)^{*}=\sum_{k=ap-bp-(p-1)/2}^{ap-bp}(-1)^k\binom{ap}{k}\binom{2ap-2k}{ap-bp-k},
\end{align*}
where $a,b$ are positive integers with $a>b$ and $p$ is an old
prime. Then
\begin{align*}
\bigg(\!\!\binom{ap}{bp}\!\!\bigg)^{*}\equiv(-1)^{ap-bp}\binom{a}{b}\pmod{p^2}.
\end{align*}
\end{cor}

The rest of the paper is arranged as follows. We shall prove
Theorems \ref{thm-a} and \ref{thm-b} in Sections 2 and 3,
respectively. The proof of Theorems \ref{thm-c}-\ref{thm-f} will be
displayed in Section 4.

\section{Proof of Theorem \ref{thm-a}}
In order to prove Theorem \ref{thm-a}, we require the following two
lemmas (cf. \cite[Lemma 3.3]{Tauraso} and \cite[Lemma 2.1]{Liu-b}).

\begin{lem}\label{lemm-a}
Let $k,n$ be positive integers satisfying $1\leq k\leq n-1$. Then
\begin{align*}
\binq{2k-1}{k}\equiv(-1)^kq^{k(3k-1)/2}\binq{n-k}{k}\pmod{\Phi_n(q)}.
\end{align*}
\end{lem}

\begin{lem}\label{lemm-b}
Let $n$ be a nonnegative integer. Then
\begin{align*}
(1-q^n)\sum_{k=0}^{\lfloor
n/2\rfloor}\frac{(-1)^kq^{k(k-1)/2}}{1-q^{n-k}}\binq{n-k}{k}=\theta_n(q).
\end{align*}
\end{lem}

Now we start to prove Theorem \ref{thm-a}.

\begin{proof}[Proof of Theorem \ref{thm-a}]
It is routine to see that
\begin{align}
\bigg(\!\!\binom{an}{bn}\!\!\bigg)_q'=\binq{an}{bn}+\sum_{k=1}^{\lfloor
n/2\rfloor}q^{k(k+bn)}\binq{an}{k}\binq{an-k}{bn+k}.
\label{eq:wei-g}
\end{align}
Noticing the relation
\begin{align*}
1-q^{an}=(1-q^n)(1+q^n+\cdots+q^{an-n})\equiv
a(1-q^n)\pmod{\Phi_n(q)^2},
\end{align*}
we can proceed as follows:
\begin{align}
\binq{an}{k}&=\frac{(1-q^{an})(1-q^{an-1})\cdots(1-q^{an-k+1})}{(1-q)(1-q^{2})\cdots(1-q^{k})}
\notag\\[2mm]
&\equiv\frac{a(1-q^{n})}{(1-q^{k})}\frac{(1-q^{-1})\cdots(1-q^{-k+1})}{(1-q)(1-q^{2})\cdots(1-q^{k-1})}
\notag\\[2mm]
&=\frac{a(1-q^{n})(-1)^{k-1}q^{-k(k-1)/2}}{1-q^{k}}
\notag\\[2mm]
&\equiv\frac{a(1-q^{n})(-1)^{k}q^{-k(k+1)/2}}{1-q^{n-k}}\pmod{\Phi_n(q)^2}.
\label{eq:wei-h}
\end{align}
With the help of Lemma \ref{lemm-a}, there holds
\begin{align}
\binq{an-k}{bn+k}&=\binq{an}{bn}\frac{(1-q^{an-bn})(1-q^{an-bn-1})\cdots(1-q^{an-bn-2k+1})}{(1-q^{an})\cdots(1-q^{an-k+1})(1-q^{bn+1})\cdots(1-q^{bn+k})}
\notag\\[2mm]
&\equiv\binq{an}{bn}\frac{a-b}{a}\frac{(1-q^{-1})(1-q^{-2})\cdots(1-q^{-2k+1})}{(1-q^{-1})\cdots(1-q^{-k+1})(1-q)\cdots(1-q^{k})}
\notag\\[2mm]
&=\binq{an}{bn}\frac{a-b}{a}(-1)^kq^{-k(3k-1)/2}\binq{2k-1}{k}
\notag\\[2mm]
 &\equiv\binq{an}{bn}\frac{a-b}{a}\binq{n-k}{k}
\pmod{\Phi_n(q)}. \label{eq:wei-i}
\end{align}
The combination of \eqref{eq:wei-g}-\eqref{eq:wei-i} gives
\begin{align*}
\bigg(\!\!\binom{an}{bn}\!\!\bigg)_q'\equiv\binq{an}{bn}+\binq{an}{bn}(a-b)(1-q^n)\sum_{k=1}^{\lfloor
n/2\rfloor}\frac{(-1)^kq^{k(k-1)/2}}{1-q^{n-k}}\binq{n-k}{k}\pmod{\Phi_n(q)^2}.
\end{align*}
Evaluating the series on the right-hand side by Lemma \ref{lemm-b},
we arrive at \eqref{eq:wei-a}.

\end{proof}

\section{Proof of Theorem \ref{thm-b}}

For proving Theorem \ref{thm-b}, we draw support from Lemmas
\ref{lemm-a} and \ref{lemm-b} and the following lemma (cf.
\cite[Lemma 2.2]{Liu-b}).

\begin{lem}\label{lemm-c}
Let $n$ be a nonnegative integer. Then
\begin{align*}
(1-q^n)\sum_{k=0}^{\lfloor
n/2\rfloor}\frac{(-1)^kq^{k(k-3)/2}}{1-q^{n-k}}\binq{n-k}{k}=\vartheta_n(q).
\end{align*}
\end{lem}

Now we start to prove Theorem \ref{thm-b}.

\begin{proof}[Proof of Theorem \ref{thm-b}]
Replacing $k$ by $an-bn-k$ in $\tau_0(an,bn,q)'$, we have
\begin{align}
\tau_0(an,bn,q)'&=\sum_{k=0}^{\lfloor n/2\rfloor}
(-1)^{an-bn-k}q^{(an-bn-k)(an+bn+k+1)/2}\binq{an}{bn+k}\binq{2bn+2k}{k}
\notag\notag\\[2mm]
&=(-1)^{an-bn}q^{(an-bn)(an+bn+1)/2}\binq{an}{bn}
\notag\\[2mm]
&\quad+\sum_{k=1}^{\lfloor n/2\rfloor}
(-1)^{an-bn-k}q^{(an-bn-k)(an+bn+k+1)/2}\binq{an}{bn+k}\binq{2bn+2k}{k}.
 \label{eq:wei-j}
\end{align}
It is not difficult to verify that
\begin{align}
\binq{an}{bn+k}&=\binq{an}{bn}\frac{(1-q^{an-bn})(1-q^{an-bn-1})\cdots(1-q^{an-bn-k+1})}{(1-q^{bn+1})(1-q^{bn+2})\cdots(1-q^{bn+k})}
\notag\\[2mm]
&\equiv\binq{an}{bn}\frac{(a-b)(1-q^{n})}{(1-q^{k})}\frac{(1-q^{-1})\cdots(1-q^{-k+1})}{(1-q)(1-q^{2})\cdots(1-q^{k-1})}
\notag\\[2mm]
&=\binq{an}{bn}\frac{(a-b)(1-q^{n})(-1)^{k-1}q^{-k(k-1)/2}}{1-q^{k}}
\notag\\[2mm]
&\equiv\binq{an}{bn}\frac{(a-b)(1-q^{n})(-1)^{k}q^{-k(k+1)/2}}{1-q^{n-k}}\pmod{\Phi_n(q)^2}.
\label{eq:wei-k}
\end{align}
In terms  of Lemma \ref{lemm-a}, there is
\begin{align}
\binq{2bn+2k}{k}&=\frac{(1-q^{2bn+2k})(1-q^{2bn+2k-1})\cdots(1-q^{2bn+k+1})}{(1-q)(1-q^{2})\cdots(1-q^{k})}
\notag\\[2mm]
&\equiv(1+q^k)\binq{2k-1}{k}
\notag\\[2mm]
&\equiv(-1)^kq^{k(3k-1)/2}(1+q^k)\binq{n-k}{k} \pmod{\Phi_n(q)}.
\label{eq:wei-l}
\end{align}
The combination of \eqref{eq:wei-j}-\eqref{eq:wei-l} produces
\begin{align*}
\tau_0(an,bn,q)'&\equiv(-1)^{an-bn}q^{(an-bn)(an+bn+1)/2}\binq{an}{bn}+(-1)^{an-bn}q^{(an-bn)(an+bn+1)/2}\binq{an}{bn}
\notag\\[2mm]
&\quad\times(a-b)(1-q^n)\sum_{k=1}^{\lfloor
n/2\rfloor}\frac{(-1)^{k}q^{k(k-3)/2}(1+q^k)}{1-q^{n-k}}\binq{n-k}{k}\pmod{\Phi_n(q)^2}.
\end{align*}
Calculating the series on the right-hand side by Lemmas \ref{lemm-b}
and \ref{lemm-c}, we deduce \eqref{eq:wei-b}.
\end{proof}

\section{Proof of Theorems \ref{thm-c}-\ref{thm-f}}
Firstly, we shall prove Theorem \ref{thm-c}.

\begin{proof}[Proof of Theorem \ref{thm-c}]
Replace $k$ by $an-bn-k$ in $T_0(an,bn,q)'$ to find
\begin{align}
T_0(an,bn,q)'&=\sum_{k=0}^{\lfloor n/2\rfloor}
(-1)^{an-bn-k}\binq{an}{bn+k}_{q^2}\binq{2bn+2k}{k}
\notag\\[2mm]
&=(-1)^{an-bn}\binq{an}{bn}_{q^2} +\sum_{k=1}^{\lfloor n/2\rfloor}
(-1)^{an-bn-k}\binq{an}{bn+k}_{q^2}\binq{2bn+2k}{k}.
 \label{eq:wei-m}
\end{align}
Similar to the derivation of \eqref{eq:wei-k}, we obtain
\begin{align}
\binq{an}{bn+k}_{q^2}&=\binq{an}{bn}_{q^2}\frac{(1-q^{2an-2bn})(1-q^{2an-2bn-2})\cdots(1-q^{2an-2bn-2k+2})}{(1-q^{2bn+2})(1-q^{2bn+4})\cdots(1-q^{2bn+2k})}
\notag\\[2mm]
&\equiv\binq{an}{bn}_{q^2}\frac{2(a-b)(1-q^{n})}{(1-q^{2k})}\frac{(1-q^{-2})\cdots(1-q^{-2k+2})}{(1-q^2)(1-q^{4})\cdots(1-q^{2k-2})}
\notag\\[2mm]
&=\binq{an}{bn}_{q^2}\frac{2(a-b)(1-q^{n})(-1)^{k-1}q^{-k(k-1)}}{(1-q^{2k})}
\notag
\end{align}
\begin{align}
&\equiv\binq{an}{bn}_{q^2}\frac{2(a-b)(1-q^{n})(-1)^{k}q^{-k^2}}{(1+q^k)(1-q^{n-k})}\pmod{\Phi_n(q)^2}.
\label{eq:wei-n}
\end{align}
Substituting \eqref{eq:wei-l} and \eqref{eq:wei-n} into
\eqref{eq:wei-m}, we get
\begin{align*}
T_0(an,bn,q)'&\equiv(-1)^{an-bn}\binq{an}{bn}_{q^2}+(-1)^{an-bn}\binq{an}{bn}_{q^2}2(a-b)(1-q^n)
\\[2mm]
&\quad\times\sum_{k=1}^{\lfloor
n/2\rfloor}\frac{(-1)^{k}q^{k(k-1)/2}}{1-q^{n-k}}\binq{n-k}{k}\pmod{\Phi_n(q)^2}.
\end{align*}
Evaluating the series on the right-hand side by Lemma \ref{lemm-b},
we catch hold of \eqref{eq:wei-c}.
\end{proof}

Secondly, we shall prove Theorem \ref{thm-d}.

\begin{proof}[Proof of Theorem \ref{thm-d}]
Replace $k$ by $an-bn-k$ in $T_1(an,bn,q)'$ to derive
\begin{align*}
T_1(an,bn,q)'&=\sum_{k=0}^{\lfloor n/2\rfloor}
(-q)^{an-bn-k}\binq{an}{bn+k}_{q^2}\binq{2bn+2k}{k}
\notag\\[2mm]
&=(-q)^{an-bn}\binq{an}{bn}_{q^2} +\sum_{k=1}^{\lfloor n/2\rfloor}
(-q)^{an-bn-k}\binq{an}{bn+k}_{q^2}\binq{2bn+2k}{k}
\notag\\[2mm]
&\equiv(-q)^{an-bn}\binq{an}{bn}_{q^2}+(-q)^{an-bn}\binq{an}{bn}_{q^2}2(a-b)(1-q^n)
\notag\\[2mm]
&\quad\times\sum_{k=1}^{\lfloor
n/2\rfloor}\frac{(-1)^{k}q^{k(k-3)/2}}{1-q^{n-k}}\binq{n-k}{k}\pmod{\Phi_n(q)^2},
\end{align*}
where we have employed \eqref{eq:wei-l} and \eqref{eq:wei-n}.
Calculating the series on the right-hand side by Lemma \ref{lemm-c},
we are led to \eqref{eq:wei-d}.
\end{proof}

For the aim to prove Theorem \ref{thm-e}, we need the following
lemma.

\begin{lem}\label{lemm-d}
Let $n$ be a nonnegative integer. Then
\begin{align}
(1-q^n)\sum_{k=0}^{\lfloor
n/2\rfloor}\frac{(-1)^kq^{k(3k-1)/2}}{1-q^{n-k}}\binq{n-k}{k}\equiv\theta_n(q^{-1})\pmod{\Phi_n(q)^2}.
\label{eq:wei-aa}
\end{align}
\end{lem}

\begin{proof}
Performing the replacement $q\to q^{-1}$ in Lemma \ref{lemm-b},
there holds
\begin{align*}
(1-q^{n})\sum_{k=0}^{\lfloor
n/2\rfloor}\frac{(-1)^kq^{k(3k-1)/2-kn}}{1-q^{n-k}}\binq{n-k}{k}
=\theta_n(q^{-1}).
\end{align*}
Considering that $q^n\equiv1\pmod{\Phi_n(q)}$, it is ordinary to
achieve \eqref{eq:wei-aa}.
\end{proof}

Thirdly, we shall prove Theorem \ref{thm-e}.

\begin{proof}[Proof of Theorem \ref{thm-e}]
Replacing $k$ by $an-bn-k$ in $t_0(an,bn,q)'$ and using
\eqref{eq:wei-l} and \eqref{eq:wei-n}, we have
\begin{align*}
t_0(an,bn,q)'&=\sum_{k=0}^{\lfloor n/2\rfloor}
(-1)^{an-bn-k}q^{(an-bn-k)^2}\binq{an}{bn+k}_{q^2}\binq{2bn+2k}{k}
\notag\\[2mm]
&=(-1)^{an-bn}q^{(an-bn)^2}\binq{an}{bn}_{q^2}
\\[2mm]
&\quad+\sum_{k=1}^{\lfloor n/2\rfloor}
(-1)^{an-bn-k}q^{(an-bn-k)^2}\binq{an}{bn+k}_{q^2}\binq{2bn+2k}{k}
\notag
\end{align*}
\begin{align*}
&\equiv(-1)^{an-bn}q^{(an-bn)^2}\binq{an}{bn}_{q^2}+(-1)^{an-bn}q^{(an-bn)^2}\binq{an}{bn}_{q^2}2(a-b)(1-q^n)
\notag\\[2mm]
&\quad\times\sum_{k=1}^{\lfloor
n/2\rfloor}\frac{(-1)^{k}q^{k(3k-1)/2}}{1-q^{n-k}}\binq{n-k}{k}\pmod{\Phi_n(q)^2}.
\end{align*}
Via Lemma \ref{lemm-d} and the last relation, we can deduce
\eqref{eq:wei-e}.
\end{proof}

For the sake of proving Theorem \ref{thm-f}, we demand the following
lemma.

\begin{lem}\label{lemm-e}
Let $n$ be a nonnegative integer. Then
\begin{align}
(1-q^n)\sum_{k=0}^{\lfloor
n/2\rfloor}\frac{(-1)^kq^{k(3k+1)/2}}{1-q^{n-k}}\binq{n-k}{k}\equiv\upsilon_n(q^{-1})\pmod{\Phi_n(q)^2}.
\label{eq:wei-bb}
\end{align}
\end{lem}

\begin{proof}
Performing the replacement $q\to q^{-1}$ in Lemma \ref{lemm-c},
there is
\begin{align*}
(1-q^{n})\sum_{k=0}^{\lfloor
n/2\rfloor}\frac{(-1)^kq^{k(3k+1)/2-kn}}{1-q^{n-k}}\binq{n-k}{k}
=\upsilon_n(q^{-1}).
\end{align*}
Observing that $q^n\equiv1\pmod{\Phi_n(q)}$, it is regular to attain
\eqref{eq:wei-bb}.
\end{proof}

Finally, we shall prove Theorem \ref{thm-f}.

\begin{proof}[Proof of Theorem \ref{thm-f}]
Replacing $k$ by $an-bn-k$ in $t_1(an,bn,q)'$ and utilizing
\eqref{eq:wei-l} and \eqref{eq:wei-n}, we catch hold of
\begin{align*}
t_1(an,bn,q)'&=\sum_{k=0}^{\lfloor n/2\rfloor}
(-1)^{an-bn-k}q^{(an-bn-k)(an-bn-k-1)}\binq{an}{bn+k}_{q^2}\binq{2bn+2k}{k}
\\[2mm]
&=(-1)^{an-bn}q^{(an-bn)(an-bn-1)}\binq{an}{bn}_{q^2}
\notag\\[2mm]
&\quad+\sum_{k=1}^{\lfloor n/2\rfloor}
(-1)^{an-bn-k}q^{(an-bn-k)(an-bn-k-1)}\binq{an}{bn+k}_{q^2}\binq{2bn+2k}{k}
\notag\\[2mm]
&\equiv(-1)^{an-bn}q^{(an-bn)(an-bn-1)}\binq{an}{bn}_{q^2}
\notag\\[2mm]
&\quad+(-1)^{an-bn}q^{(an-bn)(an-bn-1)}\binq{an}{bn}_{q^2}2(a-b)(1-q^n)
\notag\\[2mm]
&\qquad\times\sum_{k=1}^{\lfloor
n/2\rfloor}\frac{(-q)^{k}q^{k(3k+1)/2}}{1-q^{n-k}}\binq{n-k}{k}\pmod{\Phi_n(q)^2}.
\end{align*}
Through Lemma \ref{lemm-e} and the last formula, we are led to
\eqref{eq:wei-f}.
\end{proof}


\end{document}